\documentclass[11pt]{article}
\usepackage[T1]{fontenc}
\usepackage{lmodern,amsmath,amsthm,amsfonts,amssymb,graphicx,float,microtype,thmtools,underscore,mathtools,anyfontsize} 
\usepackage[usenames,dvipsnames,svgnames,table]{xcolor}
\usepackage[unicode=true]{hyperref}
\hypersetup{ 
	colorlinks,
	linkcolor={black},
	citecolor={black},
	urlcolor={blue!60!black},
	pdftitle={Notes},
	pdfauthor={Torsten Ueckerdt and David~R.~Wood and Wendy Yi}} 
\usepackage{breakurl}
\usepackage[noabbrev,capitalise]{cleveref}
\crefname{lem}{Lemma}{Lemmas}
\crefname{thm}{Theorem}{Theorems}
\crefname{cor}{Corollary}{Corollaries}
\crefname{prop}{Proposition}{Propositions}
\crefname{conj}{Conjecture}{Conjectures}
\crefname{openproblem}{Open Problem}{Open Problems}
\crefformat{equation}{(#2#1#3)}
\Crefformat{equation}{Equation #2(#1)#3}

\usepackage[shortlabels]{enumitem}
\setlist[itemize]{topsep=0ex,itemsep=0ex,parsep=0.4ex}
\setlist[enumerate]{topsep=0ex,itemsep=0ex,parsep=0.4ex}

\newcommand{\defn}[1]{\textcolor{Maroon}{\emph{#1}}}

\newcommand{\PP}{\mathcal{P}}

\newcommand{\NN}{\mathbb{N}}

\usepackage[longnamesfirst,numbers,sort&compress]{natbib}
\makeatletter
\def\NAT@spacechar{~}
\makeatother
\usepackage[bmargin=25mm,tmargin=25mm,lmargin=35mm,rmargin=35mm]{geometry}
\allowdisplaybreaks

\renewcommand{\le}{\leqslant}
\renewcommand{\geq}{\geqslant}
\renewcommand{\leq}{\leqslant}
\DeclareMathOperator{\dist}{dist}

\DeclareMathOperator{\tw}{tw}
\DeclareMathOperator{\stw}{stw}

\renewcommand{\thefootnote}{\fnsymbol{footnote}}
\allowdisplaybreaks
\theoremstyle{plain}
\newtheorem{thm}{Theorem}
\newtheorem{lem}[thm]{Lemma}
\newtheorem{cor}[thm]{Corollary}

\theoremstyle{definition}

\begin{document}
	
\author{Torsten Ueckerdt\,\footnotemark[2] \qquad David~R.~Wood\,\footnotemark[5] \qquad Wendy Yi\,\footnotemark[2] }
	
\footnotetext[2]{Institute of Theoretical Informatics, Karlsruhe Institute of Technology, Germany (\texttt{torsten.ueckerdt@kit.edu}).}
	
\footnotetext[5]{School of Mathematics, Monash   University, Melbourne, Australia  (\texttt{david.wood@monash.edu}). Research supported by the Australian Research Council.}
	
\sloppy
	
\title{\textbf{An improved planar graph product structure theorem}}
	
\maketitle
	
\begin{abstract}
Dujmovi{\'c}, Joret, Micek, Morin, Ueckerdt and Wood~[J.~ACM 2020] proved that for every planar graph $G$ there is a graph $H$ with treewidth at most 8 and a path $P$ such that $G\subseteq H\boxtimes P$. We improve this result by  replacing ``treewidth at most 8'' by ``simple treewidth at most 6''.
\end{abstract}
	
\renewcommand{\thefootnote}{\arabic{footnote}}

\section{Introduction}

This paper is motivated by the following question: what is the global structure of planar graphs? Recently, \citet{DJMMUW20} gave an answer to this question that describes planar graphs in terms of products of simpler graphs, in particular, graphs of bounded treewidth. In this note, we improve this result in two respects. To describe the result from \citep{DJMMUW20}  and our improvement, we need the following definitions. 

A \defn{tree-decomposition} of a graph $G$ is a collection $(B_x\subseteq V(G):x\in V(T))$ of subsets of $V(G)$ (called \defn{bags}) indexed by the nodes of a tree $T$, such that:
\begin{enumerate}[label=(\alph*)]
	\item for every edge $uv\in E(G)$, some bag $B_x$ contains both $u$ and $v$, and 
	\item for every vertex $v\in V(G)$, the set $\{x\in V(T):v\in B_x\}$ induces a non-empty (connected) subtree of $T$.
\end{enumerate}
The \defn{width} of a tree decomposition is the size of the largest bag minus 1. The \defn{treewidth} of a graph $G$, denoted by $\tw(G)$, is the minimum width of a tree decomposition of $G$. These definitions are due to \citet{RS-II}. Treewidth is recognised as the most important measure of how similar a given graph is to a tree. Indeed, a connected graph with at least two vertices has treewidth 1 if and only if it is a tree. See \citep{HW17,Reed03,Bodlaender-TCS98} for surveys on treewidth.

A tree-decomposition $(B_x:x\in V(T))$ of a graph $G$ is \defn{$k$-simple}, for some $k\in\NN$,  if it has  width  at most $k$, and for every set $S$ of $k$ vertices in $G$, we have $|\{x\in V(T): S\subseteq B_x\}|\leq 2$. The \defn{simple treewidth} of a graph $G$, denoted by $\stw(G)$, is the minimum $k\in\NN$ such that $G$ has a $k$-simple tree-decomposition. Simple treewidth appears in several places in the literature under various guises \citep{KU12,KV12,MJP06,Wulf16}. The following facts are well-known: A graph has simple treewidth 1 if and only if it is a linear forest. A graph has simple treewidth at most 2 if and only if it is outerplanar. A graph has simple treewidth at most 3 if and only if it has treewidth at most 3 and is planar~\citep{KV12}. The edge-maximal graphs with simple treewidth 3 are ubiquitous objects, called  \defn{planar 3-trees} or \defn{stacked triangulations} in structural graph theory and graph drawing~\citep{AP-SJADM96,KV12}, called \defn{stacked polytopes} in polytope theory~\citep{Chen16}, and called \defn{Apollonian networks} in enumerative and random graph theory~\citep{FT14}. It is also known and easily proved that $\tw(G) \leq \stw(G)\leq \tw(G)+1$ for every graph $G$ (see \citep{KU12,Wulf16}). 

The \defn{strong product} of graphs $A$ and $B$, denoted by $A\boxtimes B$, is the graph with vertex set $V(A)\times V(B)$, where distinct vertices $(v,x),(w,y)\in V(A\boxtimes B)$ are adjacent if 
(1) $v=w$ and $xy\in E(B)$, or 
(2) $x=y$ and $vw\in E(A)$, or  
(3) $vw\in E(A)$ and $xy\in E(B)$. 

\citet{DJMMUW20} proved the following theorem describing the global structure of planar graphs. 

\begin{thm}[\citep{DJMMUW20}]
\label{PlanarStructure}
Every planar graph $G$ is isomorphic to a subgraph of $H\boxtimes P$, for some planar graph $H$ with treewidth at most 8 and some path $P$.
\end{thm}

\cref{PlanarStructure} has been used to solve several open problems regarding queue layouts~\citep{DJMMUW20}, non-repetitive colourings~\citep{DEJWW20}, centered colourings~\citep{DFMS20}, clustered colourings~\citep{DEMWW}, adjacency labellings~\citep{BGP20,DEJGMM,EJM}, and vertex rankings~\citep{BDJM}.

We modify the proof of \cref{PlanarStructure} to establish the following.

\begin{thm}
\label{PlanarStructure6}
Every planar graph $G$ is isomorphic to a subgraph of $H\boxtimes P$, for some planar graph $H$ with simple treewidth at most 6 and some path $P$.
\end{thm}

\cref{PlanarStructure6} improves upon \cref{PlanarStructure} in two respects. First it is for simple treewidth (although it should be said that the proof of \cref{PlanarStructure} gives the analogous result for simple treewidth 8). The main improvement is to replace 8 by 6, which does require new ideas. The proof of \cref{PlanarStructure6} builds heavily on the proof of \cref{PlanarStructure}, which in turn builds on a result of \citet{PS18}, who showed that every planar graph has a partition into geodesic paths whose contraction gives a graph with treewidth at most 8. 

\section{Proof of \cref{PlanarStructure6}}

Our goal is to find a given planar graph $G$ as a subgraph of $H \boxtimes P$ for some graph $H$ of small treewidth and path $P$. \citet{DJMMUW20} showed this can be done by partitioning the vertices of $G$ into so-called vertical paths in a BFS spanning tree so that contracting each path into a single vertex gives the graph $H$ (see \cref{MakeProduct} and \cref{VerticalPathsProduct} below). 

To formalise this idea, we need the following terminology and notation. A \defn{partition} $\PP$ of a graph $G$ is a set of connected subgraphs of $G$, such that each vertex of $G$ is in exactly one subgraph in $\PP$. The \defn{quotient} of $\PP$, denoted $G/\PP$, is the graph with vertex set $\PP$, where distinct elements $A,B\in \PP$ are adjacent in $G/\PP$ if there is an edge of $G$ with endpoints in $A$ and $B$. Note that $G/\PP$ is a minor of $G$, so if $G$ is planar then $G/\PP$ is planar. 

If $T$ is a tree rooted at a vertex $r$, then a non-empty path $(x_0,\dots,x_p)$ in $T$ is \defn{vertical} if for some $d\geq 0$ for all $i\in[0,p]$ we have $\dist_T(x_i,r)=d+i$. 

\begin{lem}[\citep{DJMMUW20}]
\label{MakeProduct}
Let $T$ be a BFS spanning tree in a connected graph $G$. 
Let $\PP$ be a partition of $G$ into vertical paths in $T$. 
Then $G$ is isomorphic to a subgraph of $(G/\PP)\boxtimes P$, for some path $P$. 
\end{lem}

\begin{figure}[!h]
\centering
\includegraphics[width=\textwidth]{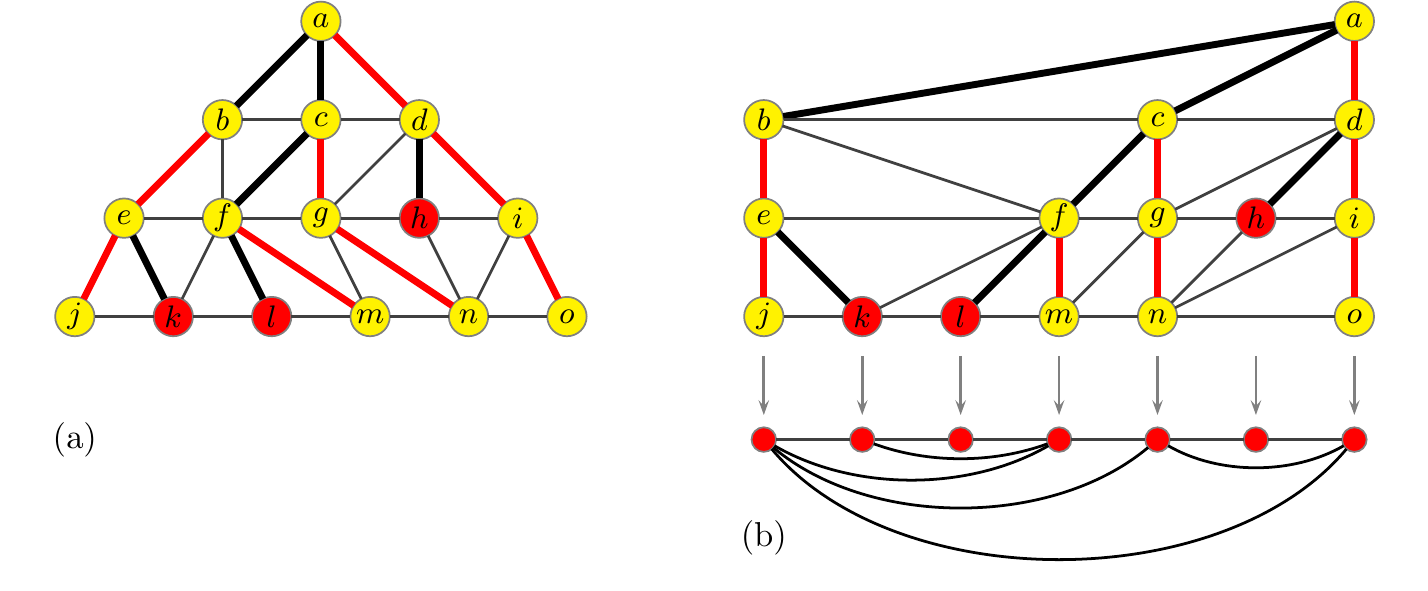}    \caption{(a) A partition $\PP$ of a planar graph $G$ into red vertical paths in a BFS spanning tree. (b) Illustration of $G$ as a subgraph of $(G/\PP)\boxtimes P$.}
\label{VerticalPathsProduct}
\end{figure}

The heart of this paper is \cref{NearTriang} below, which is an improved version of the key lemma from \citep{DJMMUW20}. The statement of \cref{NearTriang} is identical to Lemma~13 from \citep{DJMMUW20}, except that we require $F$ to be partitioned into at most $5$ instead of $6$ paths and that the tree-decomposition of $H$ is 6-simple. 

For a cycle $C$, we write $C=[P_1,\dots,P_k]$ if $P_1,\dots,P_k$ are pairwise disjoint non-empty paths in $C$, and the endpoints of each path $P_i$ can be labelled $x_i$ and $y_i$ so that $y_ix_{i+1} \in E(C)$ for $i\in[k]$, where $x_{k+1}$ means $x_1$. This implies that $V(C)=\bigcup_{i=1}^k V(P_i)$. 

The proof of \cref{NearTriang} employs the following well-known variation of Sperner's Lemma (see~\citep{Proofs4})

\begin{lem}[Sperner's Lemma]
\label{Sperner}
Let $G$ be a near-triangulation whose vertices are coloured $1,2,3$, with the outerface $F=[P_1,P_2,P_3]$ where each vertex in $P_i$ is coloured $i$. Then $G$ contains an internal face whose vertices are coloured $1,2,3$. 
\end{lem}

\begin{lem}
\label{NearTriang}
Let $G^+$ be a plane triangulation, let $T$ be a spanning tree of $G^+$ rooted at some vertex $r$ on the outerface of $G^+$, and let $P_1,\ldots,P_k$ for some $k\in[5]$, be pairwise disjoint vertical paths in $T$ such that $F=[P_1,\ldots,P_k]$ is a cycle in $G^+$. Let $G$ be the near-triangulation consisting of all the edges and vertices of $G^+$ contained in $F$ and the interior of $F$. Then $G$ has a partition $\PP$ into paths in $G$ that are vertical in $T$, such that $P_1,\ldots,P_k\in\PP$ and the quotient $H:=G/\PP$ has a 6-simple tree-decomposition such that some bag contains all the vertices of $H$ corresponding to $P_1,\ldots,P_k$. 
\end{lem}

\begin{proof}
The proof is by induction on $n=|V(G)|$. If $n=3$, then $G$ is a 3-cycle and $k\le 3$. The partition into vertical paths is $\PP=\{P_1,\ldots,P_k\}$.  The tree-decomposition of $H$ consists of a single bag that contains the $k\le 3$ vertices corresponding to $P_1,\ldots,P_k$. Now assume that $n>3$. 

We now set up an application of Sperner's Lemma to the near-triangulation $G$. We begin by colouring the vertices in $k \le 5$ colours. For $i \in \{1,\ldots,k\}$, colour each vertex in $P_i$ by $i$. Now, for each remaining vertex $v$ in $G$, consider the path $P_v$ from $v$ to the root of $T$. Since $r$ is on the outerface of $G^+$, $P_v$ contains at least one vertex of $F$. If the first vertex of $P_v$ that belongs to $F$ is in $P_i$, then assign the colour $i$ to $v$. The set $V_i$ of all vertices of colour $i$ induces a connected subgraph of $G$ for each $i\in\{1,\ldots,k\}$. Consider the graph $M = G / \{V_1,\ldots,V_k\}$ obtained by contracting each colour class $V_i$ into a single vertex $c_i$. Since $G$ is planar, $M$ is planar. (In fact, $M$ is outerplanar, although we will not use this property.)\ Moreover, if $k \geq 3$ then $[c_1,\ldots,c_k]$ is a cycle in $M$. Since $M \not\cong K_5$, we may assume without loss of generality that either $k \le 4$ or $k=5$ and $c_2c_5$ is not an edge in $M$; that is, no vertex coloured $2$ is adjacent to a vertex coloured $5$. 

Group consecutive paths from $P_1,\dots,P_k$ as follows:
\begin{itemize}
\item If $k=1$ then, since $F$ is a cycle, $P_1$ has at least three vertices, so $P_1=[v,P_1',w]$ for two distinct vertices $v$ and $w$.  Let $R_1:=v$, $R_2:=P_1'$ and $R_3:=w$.

\item If $k=2$ then, without loss of generality,  $P_1$ has at least two vertices, say $P_1=[v,P_1']$.  Let $R_1:=v$, $R_2:=P_1'$ and $R_3:=P_2$.

\item If $k=3$ then let $R_1:=P_1$, $R_2:=P_2$ and $R_3:=P_3$.

\item If $k=4$ then let $R_1:=P_1$, $R_2:=P_2$ and $R_3:=[P_3,P_4]$. 

\item If $k=5$ then let $R_1:=P_1$, $R_2:=[P_2,P_3]$ and $R_3:=[P_4,P_5]$. 
\end{itemize}

We now derive a $3$-colouring from the $k$-colouring above. 
For $i\in\{1,2,3\}$, colour each vertex in $R_i$ by $i$. Now, for each remaining vertex $v$ in $G$, consider again the path $P_v$ from $v$ to the root of $T$ and if the first vertex of $P_v$ that belongs to $F$ is in $R_i$, then assign the colour $i$ to $v$. 
Hence, for $k = 3$ we obtain exactly the same 3-colouring as above, while for $k \in\{4,5\}$ some pairs of colour classes from the $k$-colouring are merged into one colour class in the $3$-colouring. In each  case, we obtain a 3-colouring of $V(G)$ that satisfies the conditions of  \cref{Sperner}. Therefore there exists a triangular face $\tau=v_1v_2v_3$ of $G$ whose vertices are coloured $1,2,3$ respectively; see \cref{PlanarProof}.

\begin{figure}
\centering
\includegraphics[scale=0.75]{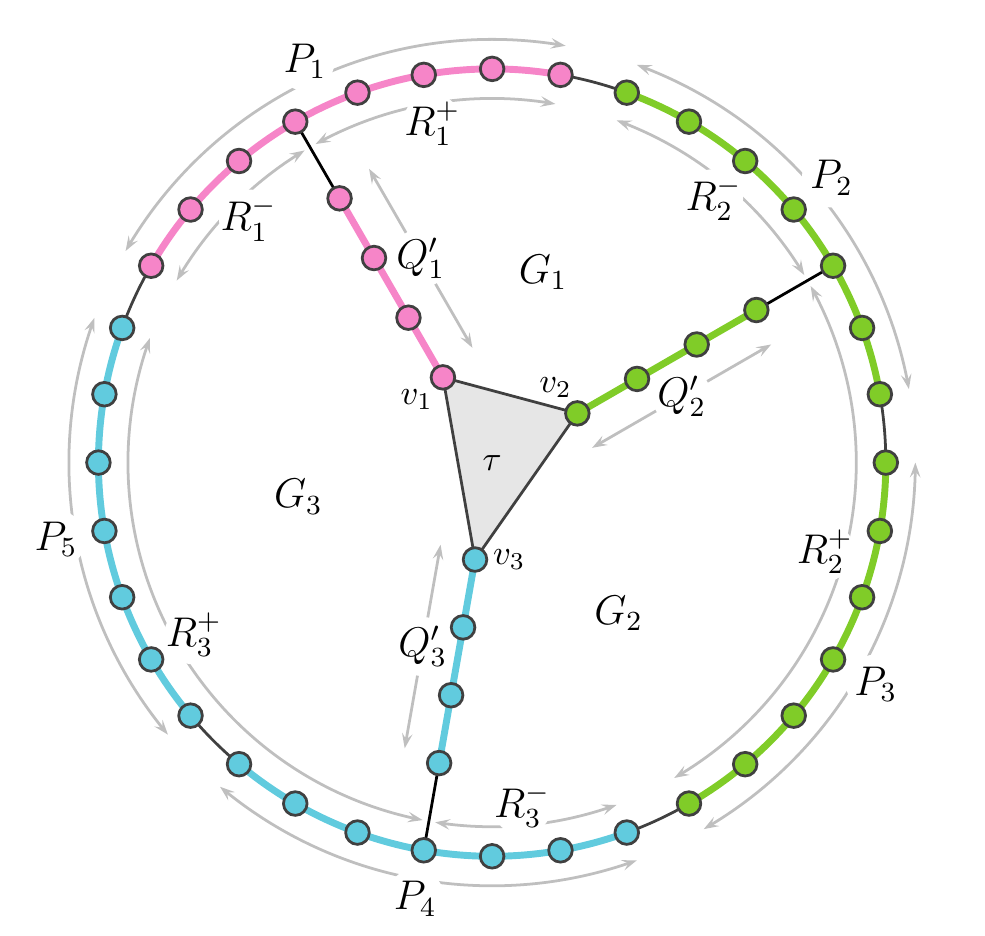}
\caption{Example of the proof of \cref{NearTriang} with $k=5$. }
\label{PlanarProof}
\end{figure}

For each $i\in\{1,2,3\}$, let $Q_i$ be the path in $T$ from $v_i$ to the first ancestor $v_i'$ of $v_i$ in $T$ that is in $F$. Observe that $Q_1$, $Q_2$, and $Q_3$ are disjoint since $Q_i$ consists only of vertices coloured $i$. Note that $Q_i$ may consist of the single vertex $v_i=v_i'$.  Let $Q_i'$ be $Q_i$ minus its final vertex $v_i'$.  Imagine for a moment that the cycle $F$ is oriented clockwise, which defines an orientation of $R_1$, $R_2$ and $R_3$. Let $R_i^-$ be the subpath of $R_i$ that contains $v'_i$ and all vertices that precede it, and let $R_i^+$ be the subpath of $R_i$ that contains $v'_i$ and all vertices that succeed it.  

Consider the subgraph of $G$ that consists of the edges and vertices of $F$, the edges and vertices of $\tau$, and the edges and vertices of $Q_1\cup Q_2\cup Q_3$. 
This graph has an outerface, an inner face $\tau$, and up to three more inner faces $F_1,F_2,F_3$ where $F_i=[Q_i',R_i^+,R_{i+1}^-,Q_{i+1}']$, where we use the convention that $Q_4=Q_1$ and $R_4=R_1$. Note that $F_i$ may be \defn{degenerate} in the sense that $[Q_i',R_i^+,R_{i+1}^-,Q_{i+1}']$ may consist only of a single edge $v_iv_{i+1}$.
  
Consider any non-degenerate $F_i=[Q_i',R_i^+,R_{i+1}^-,Q_{i+1}']$. Note that these four paths are pairwise disjoint, and thus $F_i$ is a cycle. If $Q_i'$ and $Q_{i+1}'$ are non-empty, then each is a vertical path in $T$. Furthermore, each of $R_i^+$ and $R_{i+1}^-$ consists of at most two vertical paths in $T$. Thus, $F_i$ is the concatenation of at most six vertical paths in $T$. 
Let $k_i$ be the actual number of (non-empty) vertical paths whose concatenation gives $F_i$. Then $k_1\leq 5$ and $k_3 \leq 5$ since $R_1^-$ and $R_1^+$ consist of only one vertical path in $T$. Also, if $k\le 4$ then $R_2^+$ consists of only one vertical path in $T$, implying $k_2 \le 5$. If $k=5$, then in our preliminary $k$-colouring no vertex coloured $2$ is adjacent to a vertex coloured $5$. Since $v_2v_3$ is an edge, this means that either $v_2'$ lies on $P_3$ or $v_3'$ lies on $P_4$ or both. In any case, at least one of $R_2^+$ and $R_3^-$ consists of only one vertical path in $T$, which again gives $k_2 \le 5$.

So $F_i$ is the concatenation of $k_i \le 5$ vertical paths in $T$ for each $i\in \{1,2,3\}$. Let $G_i$ be the near-triangulation consisting of all the edges and vertices of $G^+$ contained in $F_i$ and the interior of $F_i$.  Observe that $G_i$ contains $v_i$ and $v_{i+1}$ but not the third vertex of $\tau$. Therefore $G_i$ satisfies the conditions of the lemma and has fewer than $n$ vertices. By induction, $G_i$ has a partition $\PP_i$ into vertical paths in $T$, such that $H_i := G_i / \PP_i$ has a 6-simple tree-decomposition $(B^i_x:x\in V(J_i))$ in which some bag $B^i_{u_i}$ contains  the vertices of $H_i$ corresponding to the at most five vertical paths that form $F_i$.  Do this for each non-degenerate $F_i$. 

We now construct the desired partition $\PP$ of $G$. Initialise $\PP := \{P_1,\ldots,P_k\}$. Then add each non-empty $Q_i'$ to $\PP$. 
Now for each non-degenerate $F_i$, classify each path in $\PP_i$ as either  \defn{external} (that is, fully contained in $F_i$) or \defn{internal} (with no vertex in $F_i$). Add all the internal paths of  $\PP_i$ to $\PP$. By construction, $\PP$ partitions $V(G)$ into vertical paths in $T$ and $\PP$ contains $P_1,\ldots,P_k$.

Let $H:=G/\PP$. Next we construct a tree-decomposition of $H$. 
Let $J$ be the tree obtained from the disjoint union of $J_i$, taken over the $i\in\{1,2,3\}$ such that $F_i$ is non-degenerate, by adding one new node $u$ adjacent to each $u_i$. 
(Recall that $u_i$ is the node of $J_i$ for which the bag $B^i_{u_i}$ contains the vertices of $H_i$ corresponding to the paths that form $F_i$.)\ 
Let the bag $B_u$ contain all the vertices of $H$ corresponding to $P_1,\ldots,P_k,Q'_1,Q'_2,Q'_3$. 
For each non-degenerate $F_i$, and for each node $x\in V(J_i)$, initialise $B_x:= B^i_x$. 
Recall that vertices of $H_i$ correspond to contracted paths in $\PP_i$.
Each internal path in $\PP_i$ is in $\PP$.
Each external path $P$ in $\PP_i$ is a subpath of $P_j$ for some $j\in[k]$ or is one of $Q'_1, Q'_2, Q'_3$. 
For each such path $P$, for every $x\in V(J)$, in bag $B_x$, 
replace each instance of the vertex of $H_i$ corresponding to $P$ by the vertex of $H$ corresponding to the path among $P_1,\ldots, P_k, Q'_1,Q'_2,Q'_3$ 
that contains $P$.
This completes the description of $(B_x : x\in V(J))$. 
By construction, $|B_x|\leq k+3\leq 8$ for every $x\in V(J)$. 

First we show that for each vertex $a$ in $H$, the set $X:=\{x\in V(J) : a\in B_x\}$ forms a subtree of $J$. If $a$ corresponds to a path distinct from $P_1,\ldots,P_k,Q'_1,Q'_2,Q'_3$ then $X$ is fully contained in $J_i$ for some $i\in\{1,2,3\}$.
Thus, by induction $X$ is non-empty and connected in $J_i$, so it is in $J$.
If $a$ corresponds to $P$ which is one of the paths among $P_1,\ldots,P_k,Q'_1,Q'_2,Q'_3$ then 
$u\in X$ and whenever $X$ contains a vertex of $J_i$ it is because some external path of $\PP_i$ was replaced by $P$.
In particular, we would have $u_i \in X$ in that case. 
Again by induction each $X\cap J_i$ is connected and since $uu_i\in E(T)$, we conclude that $X$ induces a (connected) subtree of $J$.

Now we show that, for every edge $ab$ of $H$, there is a bag $B_x$ that contains $a$ and $b$. 
If $a$ and $b$ are both obtained by contracting any of $P_1,\ldots,P_k,Q'_1,Q'_2,Q'_3$, 
then $a$ and $b$ both appear in $B_u$.  
If $a$ and $b$ are both in $H_i$ for some $i\in\{1,2,3\}$, 
then some bag $B^i_x$ contains both $a$ and $b$. 
Finally, when $a$ is obtained by contracting a path $P_a$ in $G_i-V(F_i)$ and $b$ is obtained by contracting a path $P_b$ not in $G_i$, 
then  the cycle $F_i$ separates $P_a$ from $P_b$ so the edge $ab$ is not present in $H$. This concludes the proof that $(B_x:x\in V(J))$ is a tree-decomposition  of $H$. Note that $B_u$ contains the vertices of $H$ corresponding to $P_1,\dots,P_k$. 

By assumption the tree-decomposition $(B^i_x:x \in V(J_i))$ of $H_i$ is 6-simple for $i\in\{1,2,3\}$. Since $|B_u\cap B_{u_i}| \leq 5$ for each $i\in\{1,2,3\}$, the tree-decomposition $(B_x:x\in V(J))$ of $H$ is 6-simple, unless $|B_u|=8$, which only occurs if $k=5$ (since $|B_u|\leq k+3$). Now assume that $k=5$. 
%
Recall again that either $v_2'$ lies on $P_3$ or $v_3'$ lies on $P_4$ or both. 
Without loss of generality, $v_3'$ lies on $P_4$, and thus there is no edge between $Q_2'$ and $P_5$.

\begin{figure}
\centering
\includegraphics{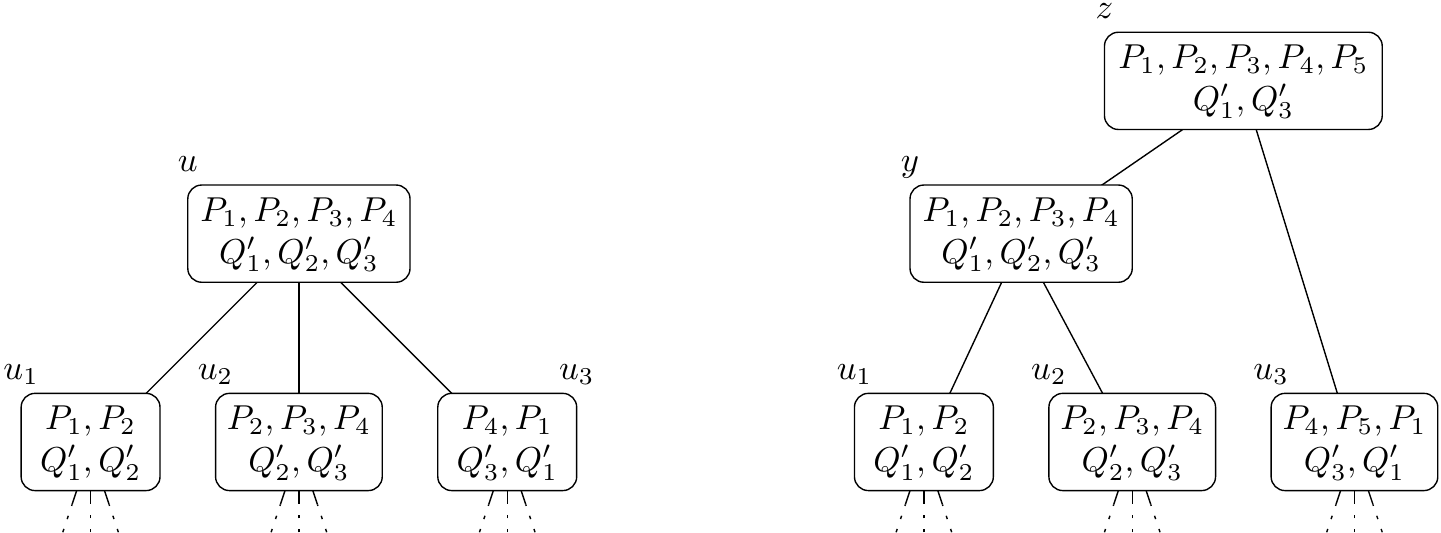}
\caption{Illustration of 6-simple tree-decomposition for a possible scenario with $k=4$ (left) and $k=5$ (right).}
\label{illustration}
\end{figure}

We now modify the above tree-decomposition of $H$ in the $k=5$ case. See \cref{illustration} for an illustration. First delete node $u$ from $J$ and the corresponding bag $B_u$. Add a new node $y$ to $J$ adjacent to $u_1$ and $u_2$, where $B_y$ consists of the vertices of $H$ corresponding to $P_1,\ldots,P_4,Q'_1,Q'_2,Q'_3$. Thus $|B_y|=7$. Add a node $z$ to $J$ adjacent to $y$ and $u_3$, where $B_z$ consists of the vertices of $H$ corresponding to $P_1,\dots,P_5,Q'_1,Q'_3$. Thus $|B_z| = 7$ and $(B_x:x\in V(J))$ is a tree-decomposition of $H$ with width 6. Since $P_5$ has no vertex in $G_1\cup G_2$, the vertex of $H$ corresponding to $P_5$ is not in $B_{u_1}\cup B_{u_2}$, and thus the nodes of $J$ whose bags contain this vertex form a connected subtree of $J$. Similarly, the vertex of $H$ corresponding to $Q'_2$ is not in $B_{u_3}$ and thus the nodes of $J$ whose bags contain this vertex form a connected subtree of $J$. The argument for the other vertices of $H$ is identical to that above. This completes the proof that $(B_x:x\in V(J))$ is a tree-decomposition of $H$ with width at most 6. It is 6-simple since the tree-decompositions of $G_1$, $G_2$ and $G_3$ are 6-simple, and $|B_y\cap B_{u_1}| \leq 5$ and $|B_y\cap B_{u_2}|\leq 5$ and $|B_z\cap B_{u_3}| \leq 5$. Moreover, $B_z$ contains the vertices of $H$ corresponding to $P_1,\dots,P_5$ as desired. 
\end{proof}

The following corollary of \cref{NearTriang} is a direct analogue of the corresponding result in \citep[Theorem~12]{DJMMUW20}.

\begin{cor}
\label{FindPartition}
Let $T$ be a rooted spanning tree in a connected planar graph $G$. Then $G$ has a partition $\PP$ into vertical paths in $T$ such that $\stw(G/\PP)\leq 6$.
\end{cor}

\begin{proof}
The result is trivial if $|V(G)|<3$. Now assume $|V(G)|\geq 3$. Let $r$ be the root of $T$. Let $G^+$ be a plane triangulation containing $G$ as a spanning subgraph with $r$ on the outerface of $G^+$. The three vertices on the outerface of $G^+$ are vertical (singleton) paths in $T$. Thus, $G^+$ satisfies the assumptions of \cref{NearTriang} with $k=3$ and $F$ being the outerface, which implies that $G^+$ has a partition $\PP$ into vertical paths in $T$ such that $\stw(G^+/\PP)\leq 6$. Note that $G/\PP$ is a subgraph of $G^+/\PP$. Hence $\stw(G/\PP)\leq 6$.
\end{proof}

\cref{FindPartition,MakeProduct} imply \cref{PlanarStructure6} (since we may assume that $G$ is connected). 

We conclude with an open problem. \citet{BDJMW} defined the \defn{row treewidth} of a graph $G$ to be the minimum integer $k$ such that $G$ is isomorphic to a subgraph of $H\boxtimes P$ for some graph $H$ with treewidth $k$ and for some path $P$. \cref{PlanarStructure} by \citet{DJMMUW20} says that planar graphs have row treewidth at most 8. Our \cref{PlanarStructure6} improves this upper bound to 6. \citet{DJMMUW20} proved a lower bound of 3. In fact, they showed that for every integer $\ell$ there is a planar graph $G$ such that for every graph $H$ and path $P$, if $G$ is isomorphic to a subgraph of $H\boxtimes P\boxtimes K_\ell$, then $H$ contains $K_4$ and thus has treewidth at least 3. What is the maximum row treewidth of a planar graph is a tantalising open problem.

{
\fontsize{10pt}{11pt}
\selectfont
\let\oldthebibliography=\thebibliography
\let\endoldthebibliography=\endthebibliography
\renewenvironment{thebibliography}[1]{%
\begin{oldthebibliography}{#1}%
\setlength{\parskip}{0.2ex}%
\setlength{\itemsep}{0.2ex}%
}{\end{oldthebibliography}}
\bibliographystyle{DavidNatbibStyle}
\bibliography{myBibliography}
}
\end{document}